\documentclass{scrartcl}
\title{Noncommutative Coverings \\ of Quantum Tori}
\author{Kay Schwieger and Stefan Wagner}
\date{}

\KOMAoptions{parskip=half, paper=a4, abstract=on}

\usepackage[latin1]{inputenc}
\usepackage{lmodern}
\usepackage{microtype}		% comment out if compilation is too slow
\usepackage[english]{babel}
\usepackage{hyperref}
\usepackage{amsmath, amsthm, amssymb, bbm, mathtools, nicefrac}
\usepackage{tablists,enumitem}
\usepackage{xspace}
\usepackage{xifthen}

\theoremstyle{plain}
	\newtheorem{thm}{Theorem}[section]
	
	\newtheorem{lemma}[thm]{Lemma}

\theoremstyle{definition}
	\newtheorem{defn}[thm]{Definition}
	\newtheorem{rmk}[thm]{Remark}

\newcommand*{\N}{\mathbb N}		% natural numbers
\newcommand*{\Z}{\mathbb Z}		% integers
\newcommand*{\R}{\mathbb R}		% real numbers
\newcommand*{\C}{\mathbb C}		% complex numbers
\newcommand*{\one}{\mathbbm 1}		% unit of an algebra
\newcommand*{\tensor}{\otimes}		% tensor product

\DeclarePairedDelimiter{\scal}{\langle}{\rangle}	% scalar product
\DeclareMathOperator{\id}{id}		% identity operator
\newcommand{\cf}{\mbox{cf.}\xspace}			% confer 
\newcommand{\eg}{\mbox{e.\,g.}\xspace}			% exempli gratia
\newcommand*{\ie}{\mbox{i.\,e.}\xspace}			% id est
\newcommand*{\ndash}{\nobreakdash-}

\newcommand*{\alg}{\mathcal}		% style for algebras 
\newcommand*{\hilb}{\mathcal} 		% style for Hilbert spaces
\newcommand*{\Star}{$^*$\ndash}
	
\DeclareMathOperator{\Ad}{Ad}		% adjunction
\DeclareMathOperator{\Aut}{Aut}
\DeclareMathOperator{\Inn}{Inn}
\DeclareMathOperator{\Out}{Out}
\DeclareMathOperator{\U}{\mathcal U}

\DeclareMathOperator{\SL}{SL}

\DeclareMathOperator{\Pic}{Pic}

\DeclareMathOperator{\Homeo}{Homeo}

\newcommand*{\aA}{\alg A}
\newcommand*{\aB}{\alg B}
\newcommand*{\Cont}{C}
\newcommand*{\cnTorus}[1][\theta]{\ensuremath{\mathbb A^n_{#1}}\xspace}
\newcommand*{\snTorus}[1][\theta]{\ensuremath{\mathbb T^n_{#1}}\xspace}
\newcommand*{\cTorus}[1][\theta]{\ensuremath{\mathbb A^2_{#1}}\xspace}
\newcommand*{\sTorus}[1][\theta]{\ensuremath{\mathbb T^2_{#1}}\xspace}
\DeclarePairedDelimiterX{\lprod}[2]{\, \prescript{}{#1}{\langle}}{\rangle}{#2}
\DeclarePairedDelimiterX{\rprod}[2]{\langle}{\rangle_{#1}}{#2}

\begin{document}

\author{
Kay Schwieger \thanks{
		iteratec GmbH, Stuttgart, 
		\href{mailto:kay.schwieger@gmail.com}{\nolinkurl{kay.schwieger@gmail.com}}
	} \and 
	Stefan Wagner \thanks{
		Blekinge Tekniska H\"ogskola,
		\href{mailto:stefan.wagner@bth.se}{\nolinkurl{stefan.wagner@bth.se}}
	}
}
%\sloppy
\maketitle

\begin{abstract}
	\noindent
	We introduce a framework for coverings of noncommutative spaces. Moreover, we study noncommutative coverings of irrational quantum tori and characterize all such coverings that are connected in a reasonable sense.

	\vspace*{0,5cm}


\end{abstract}

\section{Introduction}

In algebraic topology one obtains information about a topological situation by studying an associated algebraic situation, which are occasionally easier to investigate. The fundamental group (or Poincar\'e group) of a topological space is the first and simplest realization of this idea and turns out to be an important topological invariant. But unlike topological $K$-theory and de Rham co-homology, which both have natural noncommutative counterparts (see \eg \cite{cycliccuntz,gracia2000} an ref.  therein), the absence of a reasonable noncommutative theory of loops makes it difficult to generalize the concept of the fundamental group to the noncommutative setting. Instead, its characterization as the group of fibre preserving homeomorphisms of the universal covering (in case it exists) is of particular interest. For an approach to noncommutative coverings we may look at free actions of quantum groups on C\Star algebras, which offer a good candidate for noncommutative principal bundles (see \eg~\cite{BaCoHa15,Ell00,SchWa15,SchWa17}). From this perspective, free actions of discrete quantum groups on C\Star algebras provide a  framework for noncommutative coverings (see also \cite[Sec.~3]{Mil08} for the notion of finite \'etale coverings in algebraic geometry). 

In the present article our investigation revolves around the concrete class of quantum tori, which have been studied by many authors (see \eg \cite{Elliott86,ElliottEvans93,Elliott93,Rieffel89}) and have motivated some important work in algebraic $K$-theory such as the Pimsner-Voiculescu exact sequence for crossed products. By using ideas and methods from \cite{SchWa15,SchWa16}, we are able to construct examples of noncommutative coverings of quantum tori and to characterize all noncommutative coverings of quantum tori that are connected in a reasonable sense. More detailedly, the paper is organized as follows.

After this introduction we provide some preliminaries and notations. In particular, we discuss the quantum 2-torus $\cTorus$ in its Schr\"odinger representation and Bogoliubov transformations. The purpose of Section \ref{sec:motivation} is to motivate our approach towards noncommutative coverings and to outline some of the occurring questions. 

The main objective of Section \ref{sec:connected coverings} is to develop a reasonable noncommutative covering theory for the quantum $n$-torus $\cnTorus$. From the classical theory of coverings we know that the restriction to connected coverings is crucial. However, some typical generalizations of connectedness to noncommutative C\Star algebras are not suitable for our purposes. The most naive approach would be to call a unital C\Star algebra connected if it only contains the trivial projections $0$ and~$\one$. This would render all non-trivial finite-dimensional algebras disconnected but the quantum 2-tori $\cTorus$ for irrational $\theta \in \R$ would be highly disconnected, which is not in our interest. The next idea might be to call a C\Star algebra connected if its center contains only trivial projections. Then the quantum tori for irrational $\theta$ are connected but the discrete quantum tori, which are isomorphic to a matrix algebra, are connected, too. Since the discrete quantum tori admit a free ergodic action of the group $G = C_n \times C_n$ (see~\cite[Expl.~4.26]{Wa14}), this would have rather strange consequences for possible noncommutative notions of a universal covering or a fundamental group for the algebra~$\C$. Our approach here relies on the fact that $\cnTorus$ admits a free and ergodic gauge action of the classical $n$-torus $\mathbb{T}^n$. The corresponding infinitesimal form may be regarded as a replacement of directional derivatives and hence as an analogue of the tangent space in the classical setting (\cf \cite{CoRi87,Rieffel90}). From this point of view, the notion of connectedness that we propose in Definition~\ref{def:lift+connectedness} is motivated by the following two classical facts:
\begin{enumerate}
\item[(i)]
The tangent space of the covering space coincides with the horizontal lift of the tangent space of the base space.
\item[(ii)]
The covering space is connected if and only if every smooth function with vanishing derivative is constant.
\end{enumerate}
In this way we are able to show that each connected covering of $\cnTorus$ must necessarily have an Abelian structure group that is a finite quotient of $\mathbb{Z}^n$ (Lemma~\ref{lem:G_Abelian}). Moreover, we obtain  a characterization of connected coverings of $\cnTorus$ which is similar to the characterization of finite coverings of the classical n-torus $\mathbb{T}^n$ (Theorem~\ref{thm:main thm 2}). We conclude with a proposal for a fundamental group of~$\cnTorus$.

The aim of Section~\ref{sec:smooth abelian covering} is to study in more details noncommutative coverings of generic irrational quantum 2-tori with finite Abelian structure groups that are not necessarily connected in the sense of Definition~\ref{def:lift+connectedness} but at least smooth in a suitable way (see Definition~\ref{def:smooth covering}). Using Bogoliubov transformations we are able to show that the finite Abelian structure groups of smooth converings of $\cTorus$ are essentially the finite Abelian subgroups of $(\mathbb T / \langle \exp(2\pi\imath \theta) \rangle)^2 \rtimes \SL_2(\mathbb{Z})$ (Theorem~\ref{thm:main thm 1}). 

At this place we would like to point out that Mahanta and Suijlekom proposed in \cite{MaSu09} a fundamental group for $\cTorus$ which is based on Nori's approach to \'etale fundamental groups of smooth quasiprojective curves and involes finitely generated projective modules and Tannakian categories. Moreover, we believe that it is worth paying attention to the ideas presented in \cite{Ivan17}. There the author studies examples of noncommutative coverings similar to ours in the context of finitely generated and projective Hilbert modules. Last but not least we would like to mention Canlubos PhD thesis \cite{Can16}, which revolves around noncommutative coverings from a Hopf-Galois theoretic perspective.

\section{Preliminaries and Notations} 
\label{sec:pre+not}

Our study revolves around noncommutative coverings of quantum tori. Consequently, we blend tools from geometry, operator algebras and representation theory.

\subsection*{Classical Coverings}

Let $Y,X$ be topological spaces. A covering is a triple $(Y,X,q)$, where $q: Y\to X$ is a continuous and locally trivial map with discrete fibers. Given a covering $(Y,X,q)$, we write $\Delta(Y):=\{\phi \in \Homeo(Y) \;|\; q\circ \phi = q\}$ for the corresponding group of so-called \emph{deck transformations}. 

%A covering space $(Y,X,p)$ is said to be \emph{regular}, if $\Delta(Y)$ acts transitively on the fibers of the map $p$.

\subsection*{C$^*$-Dynamical Systems}

Let $\aA$ be a unital C\Star algebra and $G$ a topological group acting on $\aA$ by \Star automorphism $\alpha_g:\aA \to \aA$ ($g \in G$) such that $G \times \aA \to \aA$, $(g,a) \mapsto \alpha_g(a)$ is continuous. We call such a triple $(\aA,G,\alpha)$ a C\Star dynamical system. Two C\Star dynamical systems $(\aA, G, \alpha)$ and $(\aA', G', \alpha')$ are called \emph{equivalent} if there is a \Star isomorphism $\Phi:\aA \to \aA'$ and a group isomorphism $\phi:G \to G'$ such that $\Phi \circ \alpha_g = \alpha'_{\phi(g)} \circ \Phi$ for all $g \in G$.

Given a C\Star dynamical system $(\aA,G,\alpha)$ with a compact group $G$ and an irreducible representation $(\pi,V)$ of $G$, we denote by $A(\pi) \subseteq \aA$ the corresponding isotypic component. The Peter-Weyl Theorem implies that the algebraic direct sum $\smash{\bigoplus_{\pi \in \hat G} A(\pi)}$ is a dense \Star subalgebra of~$\aA$, where $\hat G$ denotes the set of equivalence classes of all irreducible representations of~$G$.

\subsection*{Freeness}

There are various non-equivalent notions of freeness of C\Star dynamical systems around the literature (see \eg \cite{EchNeOy09} or \cite{Phi09} and ref.~therein). For the preset work we say that a C\Star dynamical system $(\aA,G,\alpha)$ with a compact group $G$ is \emph{free} if the  \emph{Ellwood map}
\begin{align*}
	\Phi:\aA\otimes_{\text{alg}}\aA\rightarrow \Cont(G,\aA), \qquad \Phi(x\otimes y)(g):=x\alpha_g(y)%\label{ellwood equation}
\end{align*}
has dense range with respect to the canonical C\Star norm on $\Cont(G,\aA)$. This condition was originally introduced for actions of quantum groups on C\Star algebras by D.~A.~Ellwood~\cite{Ell00} and is known as the Ellwood condition. For equivalent formulations of freeness we may refer the reader to~\cite{SchWa15,SchWa17}.

\subsection*{The Picard Group}

Let $\aA$ be a unital C\Star algebra. We write $\Pic(\aA)$ for the Picard group, \ie, the group of equivalene classes of Morita self-equivalences over $\aA$, where the group law is given by the tensor product $[M] + [N] := [M \tensor_\aA N]$ for all Morita self-equivalences $M$, $N$ over $\aA$. Moreover, for every \Star automorphism $\alpha$ of $\aA$ we denote by $M_\alpha$ the Morita self-equivalence over $\aA$ that is given, as Banach space, by $\aA$ itself with the usual left multiplication and left inner product and the right module structure and right inner product given by $x \,. \,a := x \, \alpha(a)$ and $\rprod{\aA}{x,y} := \alpha^{-1}(x^*y)$ for all $x,y \in M_\alpha$ and $a \in \aA$. It is easily checked that the map $\alpha \mapsto [M_\alpha]$ provides a group homomorphism \mbox{$\Aut(\aA) \to \Pic(\aA)$} with the inner automorphisms as kernel, \ie, the outer automorphism group $\Out(\aA) := \Aut(\aA) / \Inn(\aA)$ may be regarded as a subgroup of $\Pic(\aA)$. We also recall that for a free C\Star dynamical system $(\aA, G, \alpha)$ with compact Abelian group $G$ the isotypic component $A(\chi)$ of a character $\chi \in \hat G$ is a Morita self-equivalence over $\aA^G$, where the inner products are given by ${}_{\aA^G}\langle x, y \rangle = xy^*$ and $\langle x,y \rangle_{\aA^G} = x^*y$ for $x,y \in A(\chi)$. The corresponding group homomorphism (see \cite[Prop.~4.4]{SchWa15})
\begin{equation}\label{eq:pic hom}
	\varphi: \hat{G} \rightarrow \Pic(\aA^G)\qquad \text{given by} \qquad \varphi(\pi):=[A(\pi)]
\end{equation}
is called the \emph{Picard homomorphism} of $(\aA, G, \alpha)$.

\subsection*{Quantum Tori}

Let $n\in\mathbb{N}$. Furthermore, let $\theta$ be a real skew-symmetric $n\times n$-matrix. The \emph{quantum $n$-torus} $\cnTorus$ is the universal C\Star algebra generated by unitaries $u_1,\ldots,u_n$ subject to the relations $u_ku_l=\exp(2\pi \imath \theta_{k\ell}) u_lu_k$ for all $1\leq k,\ell\leq n$. The smaller \Star algebra generated by all rapidly decreasing noncommutative polynomials in $u_1,\ldots,u_n$ is called the \emph{smooth quantum $n$-torus} and will be denoted by $\snTorus$. The matrix $\theta$ is called \emph{quite irrational} if, for all $\lambda \in \Z^n$, the condition $\exp(2\pi \imath \, \scal{\lambda, \theta \cdot \mu}) = 1$ for all $\mu \in \Z^n$ implies $\lambda = 0$. We point out that $\cnTorus$ is a simple C\Star algebra if $\theta$ is quite irrational. The classical $n$-torus $\mathbb T^n$ acts on $\cnTorus$ via so-called \emph{gauge transformations}, which is given on generators by $\gamma_z(u_k):=z_k \cdot u_k$ for each $z=(z_1,\ldots,z_n)\in\mathbb{T}^n$ and $1\leq k \leq n$. The corresponding C\Star dynamical system $(\cnTorus, \mathbb{T}^n, \gamma)$ is free and ergodic. 

For the special case $n=2$ we identify a skew-symmetric matrix $\theta$ with its upper right off-diagonal entry $\theta \in \R$. The quantum 2-torus $\cTorus$ admits an action of the group $\SL_2(\Z)$, given on generators $u$ and $v$ by 
\begin{align*}
	M.u = u^a v^b \qquad \text{and} \qquad M.v = u^c v^d \qquad \text{for}~M = \begin{psmallmatrix} a & b \\ c & d \end{psmallmatrix}.
\end{align*}
For each character $\lambda \in \Z^2 = \widehat{\mathbb T^2}$ this action maps the isotypic component $\cTorus(\lambda)$ with respect to the gauge action into $\cTorus(M \cdot \lambda)$. For this reason, we refer to this action as \emph{lattice transformations} of $\cnTorus$.

\subsection*{Smooth Automorphisms of the Quantum 2-Torus}
\label{sec:Bogoliubov}

Interesting for us are the so-called \emph{smooth automorphisms} of $\cTorus$, by which we mean automorphisms $\alpha:\cTorus \to \cTorus$ satisfying $\alpha(\sTorus) \subseteq \sTorus$. For an irrational number $\theta \in \R$ the group $\Aut^\infty(\cTorus)$ of all smooth automorphisms of $\cTorus$ was determinded by Elliott \cite{Elliott86} (see also \cite{Elliott93}). It is isomorphic to the semidirect product
\begin{equation*}
	\Aut^\infty(\cTorus) := P\U_0(\sTorus) \rtimes \bigl( \mathbb T^2 \rtimes \SL_2(\Z) \bigr).
\end{equation*}
Here, $P\U_0(\sTorus)$ denotes the connected component of the projective unitary group of $\sTorus$ and $\mathbb T^2$ acts via gauge transformations and $\SL_2(\Z)$ via lattice transformations. With respect to this identification the inner automorphisms in $\Aut^\infty(\cTorus)$ are precisely given by the subgroup $P\U_0(\sTorus) \rtimes \langle \exp(2\pi\imath \theta) \rangle^2$. In particular, all inner automorphisms in $\Aut^\infty(\cTorus)$ are implemented by unitaries inside of the smooth algebra \sTorus. We denote the corresponding quotient group of \emph{smooth outer automorphisms} by
\begin{equation}\label{eq:smooth}
	\Out^\infty(\cTorus) := \Aut^\infty(\cTorus) / P\U(\sTorus) = (\mathbb T / \langle \exp(2\pi\imath \theta) \rangle)^2 \rtimes \SL_2(\Z).
\end{equation}

\subsection*{Bogoliubov Transformations}\label{sec:bogo}

Let $\theta \in \R$ be irrational. For our purposes it is convenient to consider the quantum 2-torus~$\cTorus$ in its \emph{Schr\"odinger representation} on $L^2(\R)$ (see \eg~\cite{Meyer95}). In order to briefly introduce this representation let $P$ and $Q$ be self-adjoint extensions of the unbounded operators $Q \psi(s) = s \cdot \psi(s)$ and $P \psi(s) = -2 \pi \imath \theta \cdot (\partial_s \psi)(s)$, $s \in \R$, defined on all smooth functions $\psi \in L^2(\R)$. Then the family of unitaries 
\begin{align*}
	S_x := e^{-\pi \imath \theta x_1 x_2} \; e^{\imath x_1Q} \; e^{\imath x_2 P} = e^{\imath(x_1 Q + x_2 P)} 
	%\qquad 
	%\text{for}~x \in \R^2
\end{align*}
for $x \in \R^2$ generate the CCR-algebra also known as Weyl-algebra $\mathrm{CCR}(\R^2, \theta)$, which is the unique C\Star algebra generated by unitaries $S_x$, $x \in \R^2$, that satisfies the relation $S_x S_y = \exp(2\pi \imath \theta (x_1y_2 - x_2 y_1)) \; S_y S_x$ for all $x, y \in \R^2$ (see~\cite{Petz90} for details). This algebra is simple and its representation on~$L^2(\R)$ is irreducible. Due to their commutation relation, the unitaries 
\begin{align*}
u:= S_{(1,0)} = e^{\imath Q} \qquad \text{and} \qquad v := S_{(0,1)} = e^{\imath P}
\end{align*}
generate the quantum 2-torus $\cTorus$. It is easy to see that also the representation of~$\cTorus$ on~$L^2(\R)$ is irreducible. 

Our motivation for studying the Schr\"odinger representation is that all smooth automorphisms of $\cTorus$ can be implemented by unitaries on~$L^2(\R)$. The gauge transformations are provided by restricting $\Ad[S_x]$, $x \in \R^2$, to the algebra $\cTorus$. For the lattice transformations we first note that for $M = \begin{psmallmatrix} a & b \\ c & d \end{psmallmatrix}$ in $\SL_2(\R)$ the unitaries $\tilde S_x := S_{Mx}$ for $x \in \R^2$ also satisfy the commutation relation of $\mathrm{CCR(\R^2, \theta)}$. By virtue of the uniqueness result of von Neumann \cite{vNeu1931}, there is a unitary $U_M \in \mathcal B\bigl( L^2(\R) \bigr)$ with $U_M^{} S_x^{} U_M^* = S_{Mx}$ for all $x \in \R^2$. The corresponding automorphism $\beta_M := \Ad[U_M]$ is called a \emph{Bogoliubov transformation}. Restricting $\beta_M$ to the algebra $\cTorus$ then yields the lattice transformation corresponding to~$M$. The automorphisms $\Ad[u]$, $u \in \sTorus$, together with the family $\Ad[S_x]$, $x \in \R^2$, and the Bogoliubov transformations $\beta_M$, $M \in \SL_2(\Z)$, generate a group of \Star automorphisms of $\alg B(L^2(\R))$ that is isomorphic to the semi-direct product $P\U_0(\sTorus) \rtimes \bigl( \R^2 \rtimes \SL_2(\Z) \bigr)$.
%\begin{equation*}
%	\mathcal G \simeq P\U_0(\sTorus) \rtimes \bigl( \R^2 \rtimes \SL_2(\Z) \bigr).
%\end{equation*}

\pagebreak[3]
\section{Noncommutative Coverings via Free Actions}\label{sec:motivation}

In this section we motivate our approach towards noncommutative coverings and outline some of the occurring questions. First and foremost we would like to point out that our long-term goal is to find a suitable notion of a fundamental group for C$^*$-algebras. From this perspective, one of the most important features of classical covering theory is that the fundamental group of a sufficiently connected space can be identified with the group of deck transformations of the corresponding universal covering and appears in this manner as a symmetry group.

The most prominent example of a covering is defined in terms of the exponential map $\exp: \R \to \mathbb{T}$, $t \mapsto \exp(2 \pi \imath t)$. Its deck transformation group $\Delta(\R)$ is given by translations of $\R$ by integers. Since $\R$ is simply connected, it follows that the fundamental group $\pi_1(\mathbb T)$ is isomorphic to $\mathbb{Z}$. Also note that $\Delta(\R)$ acts freely and properly discontinuously. In fact, for many interesting coverings the group of deck transformations acts freely and properly discontinuously. Moreover, each free and properly discontinuous action defines a covering in terms of the corresponding quotient map.

These facts suggest to address noncommutative coverings in terms of free actions of discrete groups on C\Star algebras. Looking at the exponential map $\exp:\R \to \mathbb{T}$ we see that, despite the compact base space, the involved spaces and groups are not compact in general. Unfortunately, the subject of free actions is better understood in the compact case (see \eg \cite{BaCoHa15, SchWa15, SchWa16, SchWa17} and ref.~therein). To give a prototypical example of the problems arising from non-compactness, note that the free and properly discontinuous action which corresponds to $\exp:\R \to \mathbb{T}$ is from a C\Star algebraic point of view defined in terms of the map
\begin{align*}
	\alpha: \mathbb{Z} \times \Cont_0(\R) \to \Cont_0(\R), \qquad \alpha(\lambda,f)(t):=f(\exp(2 \pi \imath \lambda t)). 
\end{align*}
The quotient space $\R / \Z$ is clearly homeomorphic to the circle $\mathbb T$, but the relation between the C\Star algebras $\Cont_0(\R)$ and $\Cont(\mathbb T)$ is more involved. For instance, $\Cont(\mathbb T)$ is not the fixed point algebra defined by the map $\alpha$.
This suggests to extend $\alpha$ to a suitable compactification. In fact, we may look at the Bohr compactification of $\R$ or, equivalently, at the C\Star algebra $\Cont_{\mathrm{ap}}(\R)$ of almost periodic functions on $\R$, \ie, the closure of all characters of $\R$ with respect to the supremum norm. The action $\alpha$ straightforwardly extends to $\Cont_{\mathrm{ap}}(\R)$ and a few moments thought show that $\Cont(\mathbb T) = \Cont_{\mathrm{ap}}(\R)^\Z$. Furthermore, the C\Star dynamical system $\bigl( \Cont_{\mathrm{ap}}(\R), \alpha, \Z \bigr)$ may be regarded as a suitable limit of the collection of C\Star dynamical systems defined, for $m \in \mathbb{N}$, in terms of the unital C\Star subalgebra of bounded continuous function on $\R$ generated by the family of functions 
\begin{align*}
	f_{z,\lambda}(t):=z^t\exp(2 \pi \imath \lambda t), \qquad z \in C_m, \, \lambda \in \mathbb{Z},
\end{align*}
and its canonical dual action by $\Z/m\Z \cong \widehat{C_m}$, where $C_m:=\{z \in \mathbb{T} \;|\; z^m=1\}$. Here, the C\Star dynamical systems under consideration play the role of intermediate coverings of $\Cont(\mathbb T)$ with compact covering space, in which the corresponding structure groups appear as finite quotients of the fundamental group.

In order to avoid technical issues concerning compactifications, we will restrict our investigation to free actions of finite groups on unital C\Star algebras. In the classical case this corresponds to finite quotients of the fundamental group and compact covering spaces. For sake of a concise language we agree on the following convention inspired by classical covering theory: If $(\aA, G, \alpha)$ is a free C\Star dynamical system with a finite group~$G$ and fixed point algebra $\aB$, then we briefly say that $(\aA, G, \alpha)$ is a \emph{covering} of $\aB$.

\pagebreak[3]
\section{Connected Coverings of Quantum Tori}\label{sec:connected coverings}

In this section we develop a covering theory for the quantum $n$-torus $\cnTorus$. To this end first recall that $\cnTorus$ admits a free and ergodic gauge action $(\cnTorus, \mathbb{T}^n, \gamma)$. The infinitesimal form of $\gamma$ then yields an injective Lie algebra representation of $\R^n$ by \Star derivations on the smooth quantum $n$-torus $\snTorus$. This Lie algebra representation may be regarded as a replacement of directional derivatives and hence as an analogue of the tangent space in the classical case (\cf \cite{CoRi87,Rieffel90}). 

For our purposes it is convenient to view the gauge action as an action of $\R^n$. That is, using the same letter for the action, we consider the C\Star dynamical system $(\cnTorus, \mathbb{R}^n, \gamma)$ given on generators by $\gamma_s(u_k) := \exp(2 \pi \imath s_k) u_k$ for each $s=(s_1,\ldots,s_n) \in \mathbb{R}^n$ and $1 \le k \le n$. 

\pagebreak[3]
\begin{defn}\label{def:lift+connectedness}
	Let $\aA$ be a C\Star algebra containing $\cnTorus$.
	\begin{enumerate}
	\item 
		We call a C\Star dynamical system $(\aA, \mathbb{R}^n,\beta)$ a \emph{lift of the gauge action} if the kernel $\Gamma := \ker \beta$ is of maximal rank, \ie, $\R^n_\Gamma := \R^n/\Gamma$ is compact, and 
		\begin{equation*}
			\beta_s(x) = \gamma_s(x)
		\end{equation*}
		holds for all $s \in \R^n$ and $x \in \cnTorus$.
	\item
		We say that a covering $(\aA, G, \alpha)$ of $\cnTorus$ is \emph{connected} if there is an ergodic lift $(\aA, \mathbb{R}^n,\beta)$ of the gauge action such that $\alpha_g \circ \beta_s = \beta_s \circ \alpha_g$ for all $g \in G$ and $s \in \mathbb{R}^n$.
	\end{enumerate}
\end{defn}

\begin{rmk} 	\label{rmk:compact_gauge}
	If the lift of a gauge action is ergodic, then the factored system $(\aA, \R^n_\Gamma, \beta)$ is automatically free (see Lemma~\ref{lem:injective}).
\end{rmk}

Classically, every covering of $\mathbb T^n$ admits a horizontal lift of the tangent space of $\mathbb T^n$ into the tangent space of the covering. Since the group of deck transformations act properly discontinuously, the tangent space of the covering in fact coincides with the horizontal lift. In the noncommutative framework a lift of the gauge action on $\cnTorus$ replaces the horizontal lift and the notion of connectedness may then be regarded as an analogue to the classical fact that the covering is connected if and only if every smooth function with vanishing derivative is constant.

Recall that, given a subgroup $\Gamma \le \Z^n$ of maximal rank, there is an invertible $n \times n$-matrix $M$ with integer valued entries such that $\Gamma = M \cdot \mathbb{Z}^n$. The dual group of $\mathbb{R}^n_\Gamma$ is parametrized by $(M^{-1})^T \cdot \mathbb{Z}^n$, where each character of $\mathbb{R}^n_\Gamma$ is of the form 
\begin{align*}
	\chi_\lambda:\mathbb{R}^n_\Gamma \to \mathbb{T}, \qquad\chi_\lambda([s]):=\exp(2 \pi \imath \langle \lambda,s\rangle \bigr)
\end{align*}
for some $\lambda \in (M^{-1})^T \cdot \mathbb{Z}^n$. For sake of brevity we do not distinguish between $\lambda$ and~$\chi_\lambda$.

\begin{lemma} 	\label{lem:gauge_lift}
	Let $\theta$ be quite irrational and let $(\aA, \mathbb{R}^n, \beta)$ be an ergodic lift of the gauge action with $\ker(\beta)=M \cdot \mathbb{Z}^n$ for some invertible $n \times n$-matrix $M$ with integer valued entries. Then $(\aA, \mathbb{R}^n, \beta)$ is equivalent to $(\mathbb A^n_{\theta'}, \mathbb{R}^n, \gamma')$ for some skew-symmetric $n \times n$-matrix $\theta'$ satisfying 
	\begin{equation} 	\label{eq:theta_relation}
		M\theta'M^T  \in \theta + M_n(\Z)
	\end{equation}
	and the canonical gauge action $\gamma'$ on $\cnTorus[\theta']$.
\end{lemma}
\begin{proof}
	Since $\beta$ is ergodic, for each character $\lambda \in \widehat{\R^n_\Gamma} = (M^{-1})^T \cdot \Z^n$ the isotypic component of $(\aA, \mathbb{R}^n_\Gamma, \beta)$ is given by multiples of a unitary element $U(\lambda) \in A(\lambda)$. The unitary $U(\lambda)$ is unique up to a phase and hence $U(\lambda_1) U(\lambda)$ is a multiple of $U(\lambda_1 + \lambda_2)$. Since the elements $(M^{-1})^T e_k$, $1 \le k \le n$, generate $\widehat{\R^n_\Gamma}$, the unitaries $U_1, \dots, U_n$ given by $U_k := U \bigl((M^{-1})^T e_k \bigr)$ generate the C\Star algebra~$\alg A$. We fix a real skew-symmetric $n\times n$-matrix $\theta'$ such that 
	\begin{equation*}
		U_k U_\ell = \exp\bigl(2\pi \imath \theta'_{k,\ell} \bigr) \, U_\ell U_k
	\end{equation*}
	for all $1 \le k, \ell \le n$. Then a standard computation shows that 
	\begin{equation}\label{eq:important eq}
		U(\lambda_1) U(\lambda_2) = \exp \Bigl(2\pi \imath \,  \scal{\lambda_1,  M \theta' M^T \lambda_2}  \Bigr) \, U(\lambda_2) U(\lambda_1)
	\end{equation}
	for all $\lambda_1, \lambda_2 \in (M^{-1})^T \cdot \mathbb{Z}^n$. Since $\beta$ is a lift of the gauge action, the unitaries $u_k := U(e_k)$ for $1 \le k \le n$ are canonical generators of $\cnTorus \subseteq \aA$. Therefore, we obtain
	\begin{align*}
		\MoveEqLeft
		\exp(2\pi \imath \theta_{k, \ell}) \, u_\ell u_k 
		= u_k u_\ell 
		= U(e_k) U(e_\ell)
		\\
		&= \exp\Bigl(2\pi \imath \, \scal{e_k, M \theta' M^T e_\ell} \Bigr) \, U(e_\ell) U(e_k) 
		= \exp \Bigl(2\pi \imath \, \scal{e_k, M \theta' M^T e_\ell} \Bigr) \, u_\ell u_k
	\end{align*}
	for all $1 \le k, \ell \le n$. It immediately follows that $M \theta' M^T \in \theta + M_n(\Z)$; in particular, $\theta'$ is quite irrational (see~\cite[Lem.~1.8]{Phi06}). As a consequence, $\cnTorus[\theta']$ is simple and therefore $\aA$ is isomorphic to~$\cnTorus[\theta']$. Moreover, a straightforward calculation verifies that we have $\beta_s = \gamma'_{M^{-1}s}$ for all~$s \in \R^n$.
\end{proof}

\begin{lemma} 	\label{lem:G_Abelian}
	Let $(\aA, G, \alpha)$ be a connected covering of $\cnTorus$. Let $(\aA, \mathbb{R}^n, \beta)$ be the corresponding lift of the gauge action and $\Gamma:=\ker(\beta)$. Then there is a group isomorphism $\phi:G \to \Z^n /\Gamma$ such that $\alpha_g = \beta_{\phi(g)}$ for all $g \in G$. 
\end{lemma}
\begin{proof}
	We regard $\beta$ as an action of the compact group $\R^n_\Gamma$ (see Remark~\ref{rmk:compact_gauge}). For every $\pi \in \hat G$ and $\lambda \in \widehat{\R^n_\Gamma}$ let $A_\alpha(\pi)$ and $A_\beta(\lambda)$ denote the corresponding isotypic components with respect to $\alpha$ and $\beta$, respectively. Since $\beta$ is free and ergodic, each $A_\beta(\lambda)$ is a 1\ndash dimensional Hilbert space with the inner product given by $\scal{x,y} := x^* y$ for \mbox{$x,y \in A_\beta(\lambda)$}. Since $\alpha$ and $\beta$ commute, each $A_\beta(\lambda)$ is $\alpha$-invariant and hence provides a 1\ndash dimensional representation of~$G$. We write $\pi_\lambda \in \hat G$ for the corresponding class. It is straightforward to check that the resulting map
	\begin{equation*}
		\psi: \widehat{\R^n_\Gamma} \to \hat G,
		\quad
		\psi(\lambda) := \pi_\lambda
	\end{equation*}
	is multiplicative, \ie, $\psi(\lambda_1 + \lambda_2) = \psi(\lambda_1) \tensor \psi(\lambda_2)$ for all $\lambda_1, \lambda_2 \in \widehat{\R^n_\Gamma}$. Furthermore, freeness of $(\aA, G, \alpha)$ implies that each $A_\alpha(\pi)$, $\pi \in \hat G$, is non-zero. Since $A_\alpha(\pi)$ is $\beta$\ndash invariant, there is a non-trivial intersection $A_\alpha(\pi) \cap A_\beta(\lambda)$ for some $\lambda \in \widehat{\R^n_\Gamma}$. Thus $\psi$ is surjective. It follows that all irreducible representations of $G$ are 1-dimensionsal, that is, $G$ is Abelian. The kernel of $\psi$ consist of those elements $\lambda \in \widehat{\R^n_\Gamma}$ such that $A_\beta(\lambda)$ intersects the $\alpha$-fixed point algebra $\cnTorus$ non-trivially, that is, we have $\ker \psi = \Z^n$. Consequently, the homomorphism~$\psi$ factors to an isomorphism $\overline \psi: \widehat{\R^n_\Gamma} / \Z^n \to \hat G$. Passing to the corresponding dual groups, we obtain an isomorphism $\phi: G \to \Z^n / \Gamma$. Next, let $\lambda \in \widehat{\R^n_\Gamma}$. Then an easy computation shows that the action of $\alpha$ on $A_\beta(\lambda)$ is given by
	\begin{equation*}
		\alpha_g(x) = \pi_\lambda(g) x 
		= \phi(\lambda)(g) x
		= \lambda \bigl( \phi(g) \bigr) x
		= \beta_{\phi(g)}(x)
	\end{equation*}
	for all $x \in A_\beta(\lambda)$ and $g \in G$. Since the subspaces $A_\beta(\lambda)$, $\lambda \in \widehat{\R^n_\Gamma}$, are total in $\aA$, we finally conclude that $\alpha_g = \beta_{\phi(g)}$ for all $g \in G$.
\end{proof}

\begin{rmk}
	\begin{enumerate}
		\item
			The group $G$ in Lemma~\ref{lem:G_Abelian} can in fact be replaced by any compact quantum group. The definitions can be extended straightforwardly to this case and for the proof only a slight adaption of the arguments is needed.
		\item
			A similar conclusion as in Lemma~\ref{lem:G_Abelian} may be derived from Peligrads duality result for free actions of compact groups, \cite[Thm. 3.3]{Pel92}.
	\end{enumerate}
\end{rmk}

We are now ready to state and prove our main theorem. From a geometric point of view, the upshot is that the coverings of the quantum torus $\cnTorus$ are of a similar form as for the classical coverings of the torus $\mathbb{T}^n$. 

\begin{thm}\label{thm:main thm 2}
	Let $\theta$ be quite irrational. Then a covering $(\aA, G, \alpha)$ of $\cnTorus$ is connected if and only if it is equivalent to $(\cnTorus[\theta'] , (M^{-1} \cdot \mathbb{Z}^n) / \mathbb{Z}^n , \gamma')$ for some invertible $n \times n$-matrix $M$ with integer valued entries, some skew-symmetric $n\times n$-matrix $\theta'$ satisfying equation~\eqref{eq:theta_relation}, and the gauge action $\gamma'$ on $\cnTorus[\theta']$.
\end{thm}
\begin{proof}
	Lemma~\ref{lem:gauge_lift} and Lemma \ref{lem:G_Abelian} together imply that every connected covering is of the asserted form. For the converse let $M$ be an invertible $n \times n$-matrix with integer valued entries and $\theta'$ a skew-symmetric $n\times n$-matrix satisfying equation~(\ref{eq:theta_relation}). Then it is easily checked that $(\cnTorus[\theta'], (M^{-1} \cdot \Z^n) / \Z^n, \gamma')$ is a connected covering of $\cnTorus$ with an ergodic lift of the gauge action given by $(\cnTorus[\theta'], \R^n, \gamma')$, where $\gamma'$ denotes the canonical gauge action on $\cnTorus[\theta']$. In fact, the only missing condition is that the fixed point algebra 
	\begin{align*}
		(\cnTorus[\theta'])^{M^{-1} \cdot \mathbb{Z}^n}
		= \{x \in \mathbb{A}^n_{\theta'}  \;|\; \gamma'_g(x)=x \quad \forall~g \in M^{-1} \cdot \mathbb{Z}^n\}
	\end{align*}
	can be identified with $\cnTorus$. For this let $U_1,\ldots,U_n$ be unitary generators of $\mathbb{A}^n_{\theta'}$ subject to the relations $U_kU_\ell=\exp(2\pi \imath \theta'_{k\ell}) U_\ell U_k$ for all $1\leq k,\ell\leq n$. Then a few moments thought show that $(\cnTorus[\theta'])^{M^{-1} \cdot \mathbb{Z}^n}$ is generated by the unitaries
	\begin{equation*}
		u_k := U(M^T e_k):=U_1^{m_{1,k}} \cdots U_n^{m_{n,k}} \in \mathbb A^n_{\theta'}(M^T e_k), \qquad 1 \leq k \leq n,
	\end{equation*}
	where $M = (m_{i,j})_{1 \le i,j \le n}$. The claim therefore follows from
	\begin{align*}
		\MoveEqLeft
		u_k u_\ell=U(M^T e_k) U(M^T e_\ell)=\exp \Bigl(2\pi \imath \,  \scal{e_k,  \bigl(M \theta' M^T \bigr) e_\ell}  \Bigr) \, U(M^T e_\ell) U(M^T e_k)
		\\
		&\overset{\eqref{eq:theta_relation}}= 		
		\exp \Bigl(2\pi \imath \, \scal{e_k, \theta e_\ell}  \Bigr) \, U(M^T e_\ell) U(M^T e_k)=\exp(2\pi \imath \theta_{k, \ell}) \, u_\ell u_k 	
	\end{align*}
	for all $1 \le k, \ell \le n$ and the simplicity of $\cnTorus$.
\end{proof}

\begin{rmk}
	\begin{enumerate}
	\item
		Since equation~\eqref{eq:theta_relation} admits multiple solutions for a given matrix $M$ the covering theory of $\cnTorus$ is more involved than for the classical coverings of~$\mathbb T^n$. This is not unexcepted since also $\mathbb T^n$ admits noncommutative coverings by quantum tori $\cnTorus$ for $\theta \in M_n(\mathbb{Q})$.
	\item 
		Let $\theta$ be quite irrational. Then by Theorem~\ref{thm:main thm 2} all possible structure groups of connected coverings of $\cnTorus$ are, up to isomorphism, of the form $\Z^n/\Gamma$, where $\Gamma$ for some subgroups $\Gamma \le \Z^n$ of maximal rank. The subgroups $\Gamma$ are partially ordered by inclusion. We define $\pi^n_\theta$ to be the inverse limit of the corresponding quotients $\Z^n / \Gamma$. Then it is not hard to see that $\pi^n_\theta$ is in fact the profinite completion of~$\Z^n$. This limit  may be regarded as a possible fundamental group for $\cnTorus$ common in algebraic geometry (\cf \cite[Sec.~3]{Mil08} for the characterization of the fundamental group of a connected variety via finite \'etale coverings). 
	\end{enumerate}
\end{rmk}

\pagebreak[3]
\section{Smooth Coverings of $\cTorus$ }\label{sec:smooth abelian covering}

In the remaining part of this article we investigate noncommutative coverings of generic irrational quantum 2-tori that are not necessarily connected in the sense of Definition~\ref{def:lift+connectedness}. For this purpose, let throughout the following $\theta$ be a fixed irrational and non-quadratic number. Then $\cTorus$ is simple and, according to \cite{Ko97}, the Picard group $\Pic(\cTorus)$ equals its subgroup $\Out(\cTorus)$ of outer automorphisms of $\cTorus$. That is, given a covering $(\aA,G,\alpha)$ of $\cTorus$ with compact Abelian group $G$, every isotypic component $A(\chi)$, $\chi \in \hat G$, contains a unitary element $u(\chi)$ (see \cite[Rem.~5.15]{SchWa15} or \cite[Thm.~4.6]{SchWa17}). The Picard homomorphism of $(\aA, G, \alpha)$ then reads as
\begin{align*}
\varphi_\aA: \hat{G} \to \Out(\cTorus), \qquad \varphi_\aA(\chi):=\bigl[\Ad[u(\chi)]_{\mid \cTorus}\bigr],
\end{align*}
where $\Ad[u(\chi)](x):=u(\chi) x u^{*}(\chi)$ for all $x \in \aA$. Conversely, given a compact Abelian group $G$ and a group homomorphism $\varphi: \hat{G} \to \Out(\cTorus)$, one may ask whether there is a free C\Star dynamical system $(\aA,G,\alpha)$ with fixed point algebra $\cTorus$ and Picard homomorphism $\varphi_\aA = \varphi$. In light of \cite[Thm. 5.15]{SchWa15}, such a C\Star dynamical system exists if and only if a certain cohomology class associated to $\varphi$ vanishes in $\text{H}^3(\hat G, \mathbb T)$. In this case, all C\Star dynamical systems with fixed point algebra $\cTorus$ and Picard homomorphism $\varphi$ are parametrized, up to isomorphism, by $\text{H}^2(\hat G, \mathbb T)$ (see \cite[Thm.~5.8 and Cor.~5.9]{SchWa15}).

At this point it is instructive to take a closer look at the class of connected coverings of $\cTorus$. By Theorem~\ref{thm:main thm 2} each connected covering of $\cTorus$ is of the form $(\cTorus[\theta'], G, \gamma')$ for some $2\times2$-matrix $\theta'$ satisfying equation~\eqref{eq:theta_relation} and some compact Abelian group $G$ acting via gauge transformations. Each isotypic component $A'(\chi)$, $\chi \in \hat G$, is contained in the smooth quantum torus $\sTorus[\theta']$. Therefore, for any unitary element $u(\chi) \in A'(\chi)$ the automorphism $\Ad[u(\chi)]$ preserves the smooth quantum torus $\sTorus =  \sTorus[\theta'] \cap \cTorus$. That is, the Picard homomorphism maps into the group of smooth outer automorphisms $\Out^\infty(\cTorus)$.

\begin{defn}\label{def:smooth covering}
	Let $G$ be a finite Abelian group. We say that a covering $(\aA,G,\alpha)$ of~$\cTorus$ is \emph{smooth} if the corresponding Picard homomorphism $\varphi_\aA: \hat{G} \to \Out(\cTorus)$ takes values in the smooth outer automorphisms~$\Out^\infty(\cTorus)$.
\end{defn}

As we have just noticed, each connected covering is smooth. It is therefore natural to ask, whether there are smooth coverings of $\cTorus$ that are not connected. The purpose of this section is to provide an affirmative answer to this question. More precisely, given any finite Abelian group $G$ and group homomorphism $\varphi:\hat G \to \Out^\infty(\cTorus)$, we give a generic construction of a covering of $\cTorus$ with Picard homomorphism $\varphi$.

We begin our study with the following observation, whose proof we leave to the reader.

\begin{lemma} \label{Bimodule structure}
	Let $\alg A \subseteq \mathcal B(\mathcal K)$ be a concrete C\Star algebra and let $u\in \mathcal{U}(\mathcal K)$ such that $\Ad[u]$ restricts to a \Star automorphism of~$\alg A$. Then the following definitions turn $ \alg A u$ into a Morita self-equivalence over~$\alg A$:
	\begin{tabenum}[(i)]
	\item
		$a.xu:=axu$ for $a,x \in \alg A$, 
	\item
		$xu.a := xua = x\Ad[u](a)u$ for $a,x \in \alg A$,
	\\
	\item
		${}_{\alg A}\langle xu,yu\rangle:=xy^*$ for $x,y \in \aA$, 
	\item
		$\langle xu,yu\rangle_{\aA}:=\Ad[u]^*(x^*y)$ for $x,y \in \alg A$.
	\end{tabenum}
\end{lemma}

Given a \Star algebra $A$, a group $H$, and for each $h\in H$ a linear subspace $A_h$ of $A$ such that $A = \sum_{h \in H} A_h$. We say that the $A$ is a \emph{weakly $H$-graded} \Star algebra if $A_{h_1} \cdot A_{h_2} \subseteq A_{h_1 h_2}$ and $A_h^*=A_{h^{-1}}$ for all $h_1, h_2, h \in H$. Notice that we do not require the sum to be direct.

\pagebreak[3]
\begin{lemma}	\label{lem:alg_vs_delta}
	Let $\alg A_1 \subseteq \mathcal B(\hilb K)$ be a C\Star algebra with $\one_{\hilb K} \in \alg A_1$, let $H$ be a group, and let $u:H \to \U(\hilb K)$ be a map with $u(1) = \one_{\hilb K}$ such that $u(h) \alg A_1 = \alg A_1 u(h)$ for all $h\in H$. Then the following statements are equivalent:
	\begin{enumerate}[label=(\alph*)]
	\item	\label{en:alg_unitary:alg}
		The set $\alg A := \sum_{h \in H} \alg A_1 \, u(h)  \subseteq \alg B(\hilb K)$ is a weakly $H$-graded \Star algebra. 
	\item	\label{en:alg_unitary:uni}
		The element $\delta u(h_1, h_2):=u(h_1) u(h_2) u(h_1 h_2)^*$ lies in $\alg A_1$ for all $h_1, h_2 \in H$.
	\end{enumerate}
\end{lemma}
\begin{proof}
	If condition~\ref{en:alg_unitary:alg} holds, then for all $h_1, h_2 \in H$ the product $u(h_1)u(h_2)$ lies in $ \alg A_1 u(h_1h_2)$. Therefore, $\delta u(h_1,h_2) \in \alg A_1 u(h_1,h_2) u(h_1,h_2)^* = \alg A_1$, which verifies condition~\ref{en:alg_unitary:uni}. Conversely, if condition \ref{en:alg_unitary:uni} holds, it immediate that $\alg A$ is an algebra and that the multiplicative part of the grading condition holds. Since $u(1) = \one_{\hilb K}$, we further have $u(h)^* = \delta u(h^{-1}, h) u(h^{-1}) \in \alg A_1 u(h^{-1})$ for all $h \in H$. Hence, the algebra~$\alg A$ is closed under involution and the involutive part of the grading condition holds, which completes condition~\ref{en:alg_unitary:alg}.
\end{proof}

We continue with a general statement that provides a criterion for deciding whether the subspaces in Lemma~\ref{lem:alg_vs_delta} is direct, which may also be of independent interest.

\pagebreak[3]
\begin{lemma} 	\label{lem:sum direct}
	Let $\alg A_1 \subseteq \alg B(\hilb K)$ be an irreducible representation of a simple C\Star algebra with $\one_{\hilb K} \in \alg A_1$. Furthermore, let $H$ be a group and let $u:H \to \U(\hilb K)$ be a map with $u(1) = \one_{\hilb K}$ satisfying the following conditions:
	\begin{enumerate}
	\item	\label{en:lin_indep:com}
		For all $h \in H$ we have $u(h) \alg A_1 = \alg A_1 u(h)$.
	\item	\label{en:lin_indep:delta}
		For all $h_1,h_2 \in H$ we have $\delta u(h_1,h_2):=u(h_1) u(h_2) u(h_1 h_2)^* \in \alg A_1$.
	\item	\label{en:lin_indep:inner}
		For every $1 \neq h \in H$ the \Star automorphism $\alpha_h:=\Ad[u(h)]_{\mid \aA_1}$ of $\aA_1$ is not inner.
	\end{enumerate}
	Then the sum of subspaces $\sum_{h \in H} \alg A_1 u(h) \subseteq \alg B(\hilb K)$ is direct.
\end{lemma}
\begin{proof}
	We proof by induction that the sum of an arbitrary number of summands is direct, where the initial case of a single summand is trivial. For the induction let $n \in \N$ and let us assume that any collection of $n$ summands is direct. A brief computation shows that hypothesis~\ref{en:lin_indep:com} and \ref{en:lin_indep:delta} imply $u(h)^* \in  \alg A_1 u(h^{-1})$ and $u(h_1) u(h_2) \in \alg A_1 u(h_1 h_2)$ for all $h,h_1,h_2 \in H$. It therefore suffices to show that for any number of distinct elements $h_1, \dots, h_n \in H \setminus \{1\}$ we have $I := \alg A_1 \cap \bigl( \sum_{k=1}^n \alg A_1 u(h_k) \bigl) = \{0\}$. The set~$I$ is obviously a closed right ideal of $\alg A_1$ and by hypothesis~\ref{en:lin_indep:com} also a left ideal. Therefore, $J := I \cap I^*$ is a closed two-sided \Star ideal of~$\alg A_1$. Since $\alg A_1$ is simple, it follows that we have either $J = \{0\}$, in which case $I=\{0\}$ and the proof is finished, or we have $J = \alg A_1$ and hence $I = \alg A_1$. The later case leads to a contradiction as follows. Suppose $I=\alg A_1$, then we find elements $x_1, \dots, x_n \in \alg A_1$ with
	\begin{equation}
		\label{eq:lin_indep:1}
		\one_{\mathcal{K}} =  x_1 u(h_1) + \dots +  x_n u(h_n),
	\end{equation}
	and by the induction hypothesis we have $x_k \neq 0$ for all $1 \le k \le n$. Then for all $x \in \alg A_1$ we 
	obtain
	\begin{equation*}
		 x x_1 u(h_1) + \dots +  x x_n u(h_n)= x =x_1 u(h_1) x + \dots +  x_n u(h_n) x
	\end{equation*}
	On  the right hand side, each element $x_k u(h_k) x$, $1 \le k \le n$, lies in $ \alg A_1 u(h_k)$ by hypothesis~\ref{en:lin_indep:com}. Due to uniqueness of the representation in the direct sum $\sum_{k=1}^n  \alg A_1 u(h_k)$, we obtain $x x_k u(h_k)  = x_k u(h_k) x$ for all $x \in \alg A_1$. Since the representation $\alg A_1 \subseteq \alg B(\mathcal K)$ is irreducible, we may conclude that $x_k u(h_k)$ is a nonzero multiple of $\one_{\hilb K}$. That is, $u(h_k)$ is a multiple of $x_k^* \in \alg A_1$ which contradicts Hypothesis~\ref{en:lin_indep:inner}. 
\end{proof}

From now on we consider $\cTorus$ in its Schr\"odinger representation.

\pagebreak[3]
\begin{lemma}	\label{lem:automs_2_unitaries}
	Let $H$ be a group and let $\varphi:H \to \Out^\infty(\cTorus)$ be a group homomorphism. Then the following assertions hold:
	\begin{enumerate}
		\item 	\label{en:automs_2_unitaries:ex}
			There is a map $u: H \to \U \bigl( L^2(\R) \bigr)$ with $u(h) \, \sTorus = \sTorus \, u(h)$ and \mbox{$[\Ad[u(h)]_{\mid \cTorus}] = \varphi(h)$} for every $h \in H$.
		\item 	\label{en:autosm_2_unitaries:delta}
			For every map $u:H \to \U \bigl( L^2(\R) \bigr)$ as in part~\ref{en:automs_2_unitaries:ex} of this Lemma the element $\delta u(h_1,h_2):=u(h_1) u(h_2) u(h_1 h_2)^*$ lies in $\sTorus$ for all $h_1,h_2 \in H$.
	\end{enumerate}
\end{lemma}
\begin{proof}
	\begin{enumerate}
	\item 
		This is essentially a consequence of the discussion in Section~\ref{sec:Bogoliubov}.  
		We recall that $\Out^\infty(\cTorus) = (\mathbb T / \langle \exp(2\pi\imath \theta) \rangle)^2 \rtimes \SL_2(\Z)$ and write elements of $\Out^\infty(\cTorus)$ as tuples $\varphi(h) = \bigl( w(h), M(h) \bigr)$ with $w(h) \in (\mathbb T/\langle \exp(2\pi\imath \theta)\rangle)^2$ and $M(h) \in \SL_2(\Z)$, respectively. Furthermore, we denote by $q: \R^2 \to (\mathbb T / \langle \exp(2\pi\imath \theta) \rangle)^2$ the quotient map given by $q(s,t) := [(\exp(2\pi \imath s), \exp(2\pi \imath t))]$. Next, for each $h \in H$, we pick an element $x(h) \in \R^2$ with $q \bigl( x(h) \bigr) = w(h)$ and put $u(h) := S_{x(h)} U_{M(h)}$. Then it is easily verified that the resulting map $u:H \to \U \bigl( L^2(\R) \bigr)$, $h \mapsto u(h)$ fulfills the asserted properties.
	\item
		Let $u:H \to \U \bigl( L^2(\R) \bigr)$ be a map as in part~\ref{en:automs_2_unitaries:ex}. Furthermore, let $h_1, h_2 \in H$ and put $\delta u := \delta u(h_1, h_2)$ for brevity. Then the automorphism $\Ad[ \delta u ]$ can be restricted to a smooth inner automorphism of $\cTorus \subseteq \mathcal B \bigl( L^2(\R) \bigr)$. Thus, there is a unitary $v \in \sTorus$ with $(\delta u) x (\delta u)^* = v x v^*$ or, equivalently, $v^* (\delta u) x = x v^* (\delta u)$ for all $x \in \cTorus$. Since the representation $\cTorus \subseteq \mathcal B \bigl( L^2(\R) \bigr)$ is irreducible, we conclude that $v^* (\delta u)$ is a multiple of $\one$ and hence $\delta u$ lies in $\sTorus$.
	\qedhere
	\end{enumerate}
\end{proof}

We now return to the original question whether for a given finite Abelian group $G$ and a smooth group homomorphism $\varphi:\hat G \to \Out(\cTorus)$ there is a covering of $\cTorus$ with Picard homomorphism $\varphi$. Due to Lemma~\ref{lem:cocycle} we may without loss of generality assume that $\varphi$ is injective. Then by Lemma~\ref{lem:automs_2_unitaries} we find a map $u: \hat G \to \U \bigl( L^2(\R) \bigr)$ with $u(1)=\one$ implementing $\varphi$. A~few moments thought show that Lemma~\ref{lem:alg_vs_delta} and Lemma~\ref{lem:sum direct} can be applied and hence the sum
\begin{align*}
	\aA := \sum_{\chi \in \hat{G}} \cTorus \, u(\chi) \subseteq \mathcal{B}(L^2(\R))
\end{align*}
defines a $\hat G$-graded \Star algebra. Since $G$ is finite, $\aA$ is in fact a C\Star algebra. Furthermore, every summand $\cTorus \, u(\chi)$, $\chi \in \hat G$, admits a representation of $G$ given by multiplying with the corresponding character. Taking direct sums and defining the action componentwise thus yields a C\Star dynamical system $(\aA, G,\alpha)$.

\begin{thm}\label{thm:main thm 1}
	The C\Star dynamical system $(\aA, G, \alpha)$ is a covering of $\cTorus$ with Picard homomorphism $\varphi$. 
\end{thm}
\begin{proof}
	For  $\chi \in \hat G$ the isotypic component $A(\chi)$ is given by $\cTorus \, u(\chi)$. In particular, we have $\aA^G = \cTorus$. Moreover, freeness of $(\aA, G, \alpha)$ follows from the fact that every isotypic component contains a unitary element (see \cite[Rmk.~5.15]{SchWa15},  \cite[Thm.~5.8]{SchWa16}). Since, by Lemma~\ref{lem:automs_2_unitaries}, for each $\chi \in \hat G$ the class of the automorphism $\Ad[u(\chi)]$ in $\Out(\cTorus)$ coincides with $\varphi(\chi)$, we conclude that the Picard homomorphism of $(\aA, G,\alpha)$ is given by $\varphi$.
\end{proof}

\begin{rmk}
	We would like to point out that with little effort the arguments and the results presented in this section extend from finite Abelian groups to coactions of group C\Star algebras of finite groups.
\end{rmk}

\appendix
\section{On Free Group Actions}

In this short appendix we show that the underlying group homomorphism of  a free C\Star dynamical system is injective. Although it might be well-known to experts, we have not found such a statement explicitly discussed in the literature.

\begin{lemma}\label{lem:injective}
	Let $(\aA, G, \alpha)$ be a C\Star dynamical system and $\ker \alpha := \{ g \in G \;|\; \alpha_g = \id_\aA \}$.
	\begin{enumerate}
	\item 
		If $(\aA, G, \alpha)$ is free, then $\ker \alpha$ is trivial.
	\item
		If $G$ is Abelian and $\aA^G = \C \one$, then $\ker \alpha$ is trivial if and only if $(\aA, G, \alpha)$ is free.
	\end{enumerate}
\end{lemma}
\begin{proof}
	\begin{enumerate}
	\item 
		Let $K := \ker \alpha$. Then the restricted C\Star dynamical system $(\aA, K, \alpha_{\mid K})$ is free \cite[Prop.~3.17]{SchWa15} and hence the Ellwood map has dense range. Since $\alpha_g = \id_\aA$ for all $g \in K$, we may immediately conclude that $K$ must be trivial. 
	\item
		Since $\aA^G = \C \one$, each isotypic component $A(\chi)$, $\chi \in \hat G$, is a 0- or 1-dimensional and $(\aA, G, \alpha)$ is free if and only if every $A(\chi)$, $\chi \in \hat G$, is nonzero. Let $N$ be the set of characters $\chi \in \hat G$ for which $A(\chi)$ is non-zero. Then a few moments thought show that $N$ is a subgroup of $\hat G$. Moreover, the isotypic components $A(\chi)$ with $\chi \in N$ are total in $\aA$. It follows that, for each $g \in G$, we have $\alpha_g = \id_\aA$ if and only if $\chi(g) = 1$ for all $\chi \in N$. That is, $\ker \alpha$ is the dual of $\hat G/N$. Thus $\ker \alpha$ is trivial if and only if $N = \hat G$, which happens if and only if $(\aA, G, \alpha)$ is free.
	\qedhere
	\end{enumerate}
\end{proof}

\begin{lemma}\label{lem:cocycle}
	Let $N$ and $H$ be compact Abelian groups, $\omega: H \times H \to N$ a continuous symmetric 2-cocycle and $G := N \times_\omega H$ the corresponding central extension. Furthermore, let $(\aA, N, \alpha)$ be a free C\Star dynamical system with Picard homomorphism $\varphi: \hat N \to \Pic(\aA^N)$. Consider the C\Star algebra $\aA' := \Cont(H, \aA) = \Cont(H) \tensor \aA$ endowed with the action
	\begin{equation*}
		\bigl(\alpha'_{(n,h)}f\bigr)(h') := \alpha_{n + \omega(h',h)} \bigl( f(h' + h) \bigr)
	\end{equation*}
	for $(n,h) \in G$ and $h'\in H$. Then $(\aA', G, \alpha')$ is a free C\Star dynamical system with fixed point algebra $\one \tensor \aA^N \cong \aA^N$ and Picard homomorphism given by 
	\begin{equation*}
		\varphi': \hat G \to \Pic(\aA^N), \qquad \varphi'(\chi) = \varphi(\chi_N),
	\end{equation*}
	where $\chi_N(n) := \chi(n,1)$ for all $n \in N$.
\end{lemma}
\begin{proof}
	It is immediate that $\alpha'$ is strongly continuous and respects the involution. Moreover, by a standard computation the cocycle condition of $\omega$ implies that $\alpha'$ is a homomorphism. Next, let $\chi \in \hat G$. We denote by $A'(\chi)$ and $A(\chi_N)$ the isotypic components of $\alpha'$ and $\alpha$, respectively. Then a few moments thought show that $f \in \Cont(H, \aA)$ belongs to $A'(\chi)$ if and only if
	\begin{equation*}
		\alpha_{n+\omega(h',h)} \bigl( f(h' + h) \bigr) = \chi(n,h) f(h')
	\end{equation*}
	for all $n \in N$ and $h,h' \in H$. Putting $n:=0$ and $h' := 0$ we find that $f(h) = \chi(0,h) f(0)$. That is, we have $f = \chi_H \tensor f(0)$ for some $f(0) \in \aA$, where $\chi_H(h) := \chi(0,h)$ for $h \in H$. Putting $h := h' := 0$ we see that $f(0) \in A(\chi_N)$. A quick check then verifies 
	\begin{equation*}
		A'(\chi) = \chi_H \tensor A(\chi_N).
	\end{equation*}
	This shows that $\one \tensor \aA^N \cong \aA^N$ is the fixed point algebra of $(\aA', G, \alpha')$ and that $A'(\chi)$ and $A(\chi_N)$ are isomorphic Morita self-equivalences over $\aA^N$. In particular, we conclude that $(\aA', G, \alpha')$ is free and that $\varphi'(\chi) = [A'(\chi)] = [A(\chi_N)] = \varphi(\chi_N)$ for all $\chi \in \hat G$.
\end{proof}

\section*{Acknowledgement}

The first name author would like to acknowledge \emph{iteratec GmbH}. The second name author would like to express his greatest gratitude to
\emph{Carl Tryggers Stiftelse f\"or Vetenskaplig Forskning} for supporting this research and to Blekinge Tekniska H\"ogskola (BTH) for providing an excellent research environment.

\bibliographystyle{abbrv}
\bibliography{short,NC2T}

\end{document}